\documentclass[12pt,reqno]{amsart}
\usepackage{amsmath}
\usepackage{amsthm}
\usepackage{amssymb}
\usepackage{amscd}
\usepackage{amsbsy}
\usepackage{setspace}
\usepackage{graphicx}
\usepackage{color}
\usepackage{calc}
\usepackage{subfigure}

\newtheorem {theorem} {Theorem} 
\newtheorem {proposition} [theorem] {Proposition}

\newtheorem {lemma} [theorem] {Lemma}
\newtheorem {definition} {Definition}
\newtheorem {remark}[theorem]{Remark}

\title[Dynamics of a generalized Rayleigh system]{Dynamics of a generalized Rayleigh system}

\author[Ma\'ira Baldissera] {Maíra Duran Baldissera$^*$}
\address{$^*$Departamento de Matem\'atica, Instituto de Ci\^encias Matem\'aticas e Computa\c c\~ao, Universidade de S\~ao Paulo, Avenida Trabalhador S\~ao Carlense, 400, 13566-590, S\~ao Carlos, SP, Brazil.}
\email{maira.baldissera@usp.br,regilene@icmc.usp.br}

\author[Jaume Llibre] {Jaume Llibre$^{**}$}
\address{$^{**}$Departament de Matem\`{a}tiques,
Universitat Aut\`{o}noma de Barcelona, 08193,  Bellaterra, Barcelona,
Catalonia, Spain}\email{jllibre@mat.uab.cat}

\author[Regilene Oliveira]{Regilene Oliveira$^*$}



\begin{document}

\begin{abstract} Consider the first order differential system given by
 \begin{equation*}
 \begin{array}{l}
   \dot{x}= y,  \qquad \dot{y}= -x+a(1-y^{2n})y,
\end{array}
 \end{equation*}
where $a$ is a real parameter and the dots denote derivatives with respect to the time $t$. Such system is known as the generalized Rayleigh system and it appears, for instance, in the modeling of diabetic chemical processes through a constant area duct, where the effect of adding or rejecting heat is considered.  In this paper we characterize the global dynamics of  this generalized Rayleigh system. In particular we prove the existence of a unique limit cycle when the parameter $a\ne 0$. 
\end{abstract}
\maketitle

{\bf The differential system here studied modelizes a diabetic chemical processes through a constant area duct, where the effect of adding or rejecting heat is considered. Always is important to know if the model exhibits or not periodic motions. This model depens on a paramater, when this parameter is sufficiently small but non-zero it is knwon that the system has a limit cycle (an isolated periodic orbit), but it is not knwon its uniqueness, its kind of stability and all the values of the parameteer for which such a limit cycle exist. Here we solve all these questions. Moreover we describe all the global dynamics of the model controlling the orbits which come from or escape at infinity.}

\section{Introduction}
	Consider the non-linear second order differential equation 
\begin{equation} \label{eq001}
    \stackrel{..}{x}+x=a(1-\dot{x}^{2n})\dot{x},
 \end{equation}
 where  $n$  is a positive integer and $a\in \mathbb{R}$ and the dots denote derivatives with respect to the time $t$. This equation is equivalent to the following first order differential system
 \begin{equation} \label{eq1}
 \begin{array}{l}
   \dot{x}=   y,  \\
    \dot{y}=  -x+a(1-y^{2n})y,
\end{array}
 \end{equation}
 that is known as the {\it generalized Rayleigh differential equation} (the Rayleigh equation was introduced by Lorde Rayleigh in \cite{art8}, for more details about the differential system \eqref{eq001} see \cite{art3}). Such system appears, for instance, in the modeling of diabetic chemical processes through a constant area duct, where the effect of adding or rejecting heat is considered.  
 
 The aim of this study is to investigate the global dynamics of system \eqref{eq001} in the Poincaré disc, and in particular, investigate the existence and uniqueness of the limit cycles of the differential system \eqref{eq001} for diferent values of $a\in \mathbb{R}$. Applying the change of variables $(x,y,t)\to  (y,x,-t)$, system \eqref{eq1} becomes
\begin{equation} \label{eq2}
\begin{array}{l}
\dot{x}=  y+a(x^{2n}-1)x, \\
\dot{y}=  -x,
\end{array} 
\end{equation}
that is a $1$-parameter Liénard differential system already studied by Lins-Neto, Melo and Pugh in \cite{art2} and by López and Martínez in \cite{art3}. The main results of this paper are the following.

\begin{theorem} \label{teo1} The generalized Rayleigh differential system \eqref{eq1} has a unique limit cycle if $a\neq 0$. When the limit cycle exists it is stable if $a<0$ and unstable if $a>0$. For $a=0$ the system is a linear center.
\end{theorem}

\begin{theorem} \label{teo2} For the generalized Rayleigh differential system \eqref{eq1} we obtain three non-topologically equivalent phase portraits on the Poincar\'e disc, they are described in Figure \ref{phaseportraits}. The bifurcation set of the phase portraits in the parameter space is $a=0$.
\end{theorem}

\begin{figure}
    \centering
    \includegraphics[scale=0.8]{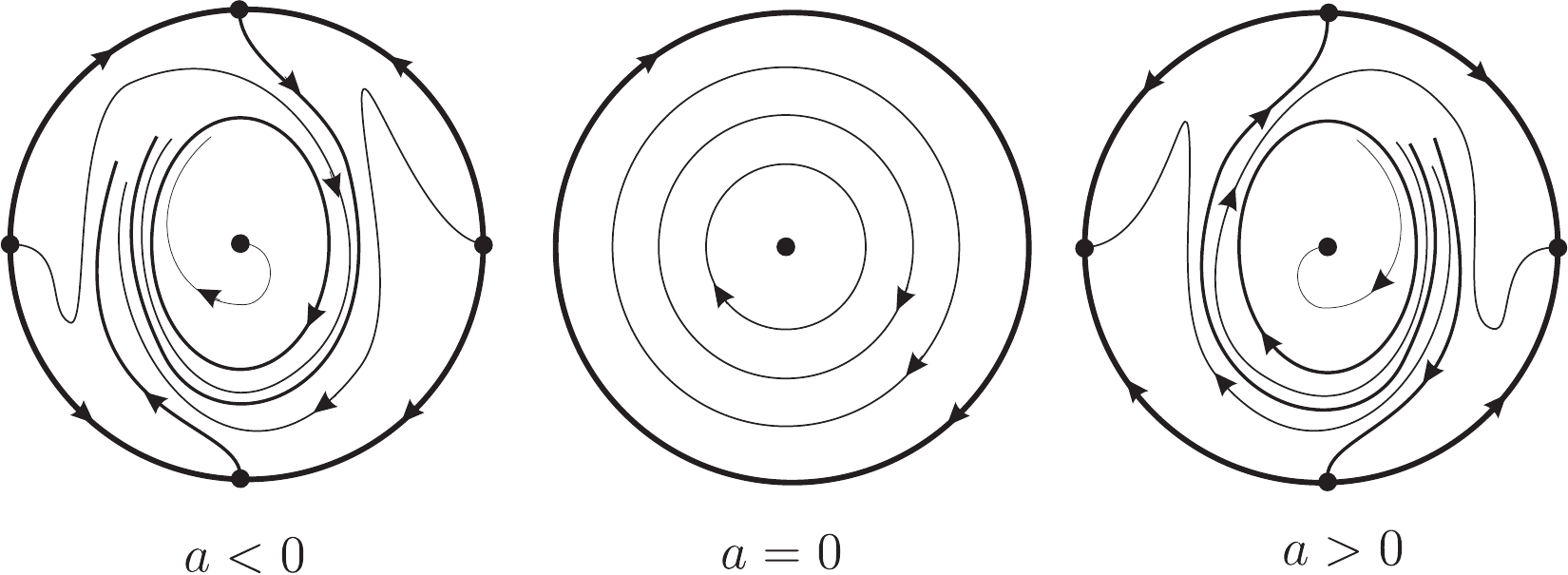}
    \caption{Phase portraits of the generalized Rayleigh systems in the Poincar\'e disc.}
    \label{phaseportraits}
\end{figure}

\section{Preliminary results} \label{preliminar}

\subsection{Basic concepts} 
Let $\mathbb{R}[x,y]$ be the set of all polynomials in the variables $x$ and $y$ with real coefficients. Associated to the polynomial differential system
\begin{equation}
 \begin{array}{cc}
        \dot x = P(x,y), &  \\
         \dot y = Q(x,y), & 
    \end{array}
\end{equation}
where $P,Q\in \mathbb{R}[x,y]$, it is the vector field $\mathcal{X}(x,y)=P \frac{\partial}{\partial x} + Q \frac{\partial}{\partial y}$ or, simply $\mathcal {X} = (P,Q)$ in $\mathbb{R}^2$. The degree of the vector field $\mathcal{X}$ is $n=\max(\deg{P},\deg{Q})$. 

The orbit $\gamma_p$ of the vector field $\mathcal{X}$ through the point $p$ is the image of the maximal solution $\phi_p: I_p \in \mathbb R \to \mathbb R^2$ endowed with an orientation, moreover, only one of the following statements holds.
\begin{enumerate}
    \item $\phi_p$ is a injection onto its image (i.e. the orbit is in correspondence with an interval of $\mathbb R$;
    \item $I_p=\mathbb{R}$ and $\phi_p$ is constant (i.e. the orbit is a point);
    \item $I_p=\mathbb{R}$ and $\phi_p$ is a periodic function, i.e, there is a constant $\tau$ such that $\phi_p(t+\tau)=\phi_p(t)$, for any $t$ and $\phi_p$ is non constant (in this case the orbit is diffeomorphic to a circle).
\end{enumerate}

\begin{definition} 
1. A singular point (or a singularity) of the vector field $X=(P,Q)$ is a point $p\in \mathbb{R}^2$ such that $\phi_p$ is constant i.e. $P(p)=Q(p)=0$. The points of $\mathbb R^2$ which are not singular are called regular points.

2. A periodic orbit of the vector field $X=(P,Q)$ is an orbit diffeomorphic to a circle.
\end{definition}

An isolated periodic orbit in the set of all periodic orbits of a planar differential system is a {\it limit cycle}. In recent years a variety of methods were used to investigate the existence of limit cycles. In this paper we shall see that the Poincar\'e-Bendixson Theorem could be an useful tool in such investigation. 

Let $p$ be a singular point of the vector field $\mathcal{X}=(P,Q)$. Denote by $\Delta(p) =P_x(p)Q_y(p)-P_y(p)Q_x(p)$ and $\Gamma =P_x(p)+Q_y(p)$, where $F_a$ denotes the derivative of the function $F(a,b)$, $a,b\in \mathbb{R}$, with respect to the variable $a$. 

\begin{definition}
An isolated singular point $p$ is said to be non-degenerate if $\Delta(p) \neq0$.
\end{definition}

It is known that $p$ is a \textit{saddle} if $\Delta<0$, a \textit{node} if $\Gamma^2-4\Delta>0$ (stable if $\Gamma <0$, unstable if $\Gamma >0$), a \textit{focus} if $\Gamma^2-4\Delta <0$ (stable if $\Gamma <0$, unstable if $\Gamma >0$), and either a \textit{weak focus or a center} if $\Gamma=0$ and $\Delta>0$ (see Theorem 2.15 of \cite{livro1} for more details about non-degenerate singular points).

\begin{definition}
An isolated singular point $p$ of $\mathcal {X}=(P,Q)$ is a hyperbolic singular point if the eigenvalues of the Jacobian matrix $J\mathcal {X}(p)$, defined by the first derivative of $P$ and $Q$ with respect to the variables $x$ and $y$, have both nonzero real part.
\end{definition}

We note that the hyperbolic singular points are non-degenerate singular points except by the weak foci and centers.

\begin{definition}
If $p$ is an isolated singular point of the vector field $\mathcal {X}$ such that $\Delta(p)=0$ and $\Gamma \neq 0$ then $p$ is called a semi-hyperbolic singular point.
\end{definition}

The next result describes the local phase portraits of the semi-hyperbolic singular points. For more details about the next proposition see Theorem 2.19 of \cite{livro1}.

\begin{proposition} \label{semihyperbolic}
Let $(0,0)$ be an isolated singular point of the vector field $\mathcal{X}$ given by
\begin{equation*}
\begin{array}{l}
         \dot x = A(x,y), \\
         \dot y = \lambda y+ B(x,y),
    \end{array}
\end{equation*}
where $A$ and $B$ are analytic in a neighborhood of the origin starting with, at least, degree $2$, in the variables $x$ and $y$. Let $y=f(x)$ be the solution of the equation $\lambda y+B(x,y)=0$ in a neighborhood of the point $(0,0)$, and suppose that the function $g(x)=A(x,f(x))$ has the expression $g(x)=ax^{\alpha}+o(x^{\alpha})$, where $\alpha \geq 2$ and $a\neq 0$. So, when $\alpha$ is odd, $(0,0)$ is either an unstable node or saddle, depending if $a>0$ or $a<0$, respectively. In the saddle's case, the separatices are tangent to the $x$-axis. If $\alpha$ is even, the $(0,0)$ is a saddle-node, i.e., the singular point is formed by the union of two hyperbolic sectors with one parabolic sector. The stable separatrix is tangent to the positive (respectively, the negative) $x$-axis at $(0,0)$, according to $a<0$ (respectively, $a>0$). The two unstable separatices are tangent to the $y$-axis at $(0,0)$.
\end{proposition}

\subsection{Limit sets} Let $\phi_p(t)=(x(t), y(t))$ be the maximal solution of a planar differential system defined on the interval $(\alpha_p, \omega_p)$ such that
$\phi_p(0)=p$. If $\omega_p = \infty$ we define the $\omega$--limit
set of $p$ as
\begin{equation*}
\omega(p)=\{q \in \mathbb R^2:\exists\{t_n\} \mbox{
with } t_n \to \infty \mbox{ and } \phi_p(t_n)\to q \mbox{ when } n
\to \infty \}.
\end{equation*}
 
Analogously, if $\alpha_p = - \infty$ we define the
$\alpha$--limit set of $p$ as
\begin{equation*}
\alpha(p)=\{q \in \mathbb R^2:\exists \{t_n\} \mbox{
with } t_n \to - \infty \mbox{ and } \phi_p(t_n)\to q \mbox{ when }
n \to \infty \}.
\end{equation*}

In what follows we denote by $\gamma_p^{+}$ the positive semi-orbit of a differential vector field $\mathcal{X}$ passing through  the point $p$. The next result is known as Poincar\'e-Bendixson Theorem and it can be used to guarantee the existence of periodic orbits in compact sets, for a proof see Theorem 1.25 of \cite{livro1}.

\begin{theorem}\label{bendixson} Let $\phi_p(t)$ be the solution of the vector field $\mathcal{X}$ defined for all $t\geq 0$ and such that $\gamma_p^+= \{\phi_p(t) : t\ge 0\}$ is contained in a compact set $K$. Assume that the vector field $\mathcal {X}$ has at most a finite number of singular points in $K$. Then one of the following statements holds.
\begin{enumerate}
    \item If $\omega(p)$ contains only regular points then $\omega(p)$ is a periodic orbit.
    \item If $\omega(p)$ contains both regular and singular points, then $\omega(p)$ is formed by a set of orbits, every one of which tends to one singular point in $\omega(p)$ when $t \to \pm \infty$.
    \item If $\omega(p)$ does not contain regular points, then $\omega(p)$ is a unique singular point.
\end{enumerate}
\end{theorem}

\subsection{Poincar\'e compactification} In this section we describe the behavior of the orbits of the planar vector field studied in the so called Poincaré disc. Roughly speaking, the Poincar\'e disc is the unitary and closed disc centered at the origin of the coordinates. The interior of such disc is identified with the plane $\mathbb R^2$, the circle $\mathbb S^1$, the boundary of the disc, is identified with the infinity of the plane. Consequently, in $\mathbb R^2$ we can go to infinity in as many directions as points in  $\mathbb S^1$. See Chapter 5 of \cite{livro1} for details about the Poincar\'e compactification.

Identifying the plane $\mathbb{R}^2$ with the plane of $\mathbb{R}^3$ of the form $(y_1,y_2,y_3)=(x_1,x_2,1)$, consider the sphere $\mathbb{S}^2=\{y\in \mathbb{R}^3:y_1^2+y_2^2+y_3^2=1\}$, called the Poincaré sphere. $\mathbb S^2$ is tangent to $\mathbb{R}^2$ in the point $(0,0,1)$. The sphere $\mathbb S^2$ can be divided in three pieces: $H_+$, $H_-$ and $S^1$, where $H_+=\{y\in \mathbb{S}^2:y_3>0\}$ the upper hemisphere, $H_-=\{y\in \mathbb{S}^2:y_3< 0\}$ is the lower hemisphere and  $\mathbb S^1=\{y\in \mathbb{S}^2:y_3=0\}$ is the equator of the sphere $\mathbb S^2$. 
 
Consider the central projection of the vector field $\mathcal{X}$ defined in $\mathbb{R}^2$ over $\mathbb{S}^2$ given by the central projections $f^+:\mathbb{R}^2\rightarrow \mathbb{S}^2$ and $f^-: \mathbb{R}^2\rightarrow \mathbb{S}^2$. More precisely, $f^+$ (respectively, $f^-$) is the intersection of the straight line that passes through the point $y$ and the origin with the upper hemisphere  (respectively, lower) of the sphere $\mathbb{S}^2$. The central projections can be written as
\begin{equation*}
     f^+(x)=\left(\frac{x_1}{\Delta(x)},\frac{x_2}{\Delta(x)},\frac{1}{\Delta(x)}\right),\qquad
     f^-(x)=\left(-\frac{x_1}{\Delta(x)},-\frac{x_2}{\Delta(x)},-\frac{1}{\Delta(x)}\right),
\end{equation*} 
where $\Delta(x)=\sqrt{x_1^2+x_2^2+1}$.

From the central projections we obtain vector fields induced in the upper and lower hemisphere of the sphere. The vector field induced in $H_+$ is defined by $$\overline{X}(y)=Df^+(x)X(x),$$ where $y=f^+(x)$, and in $H_-$ is given by $\overline{X}(y)=Df^-(x)X(x)$, where $y=f^-(x)$. Notice that $\overline{X}$ is a vector field in $\mathbb{S}^2{\setminus}\mathbb{S}^1$ that is tangent to $\mathbb{S}^2$ in each point.

Multiply the induced vector field $\overline{X}$ by the factor $\rho(x)=y_3^{d-1}$, with $d$ equal to degree of the polynomial vector field$X$, the obtained vector field is a vector field defined in whole sphere $\mathbb{S}^2$ and it is known as the Poincaré compactification of the vector field $\mathcal{X}$ in $\mathbb{S}^2$. We denote it by $p(\mathcal{X})$. 

In order to describe the analytic expression of the vector field $p(\mathcal{X})$ we consider the local charts 
$U_k=\{y\in \mathbb{S}^2:y_k>0\}$, $V_k=\{y\in\mathbb{S}^2:y_k<0\},$
for $k=1,2,3$ and the corresponding local maps $\phi_k:U_k\rightarrow \mathbb{R}^2$ and $\psi_k:V_k\rightarrow \mathbb{R}^2$ defined by $\phi_k(y)=\psi_k(y)=(y_m/y_k,y_n/y_k),$ for $m<n$ and $m,n\neq k$. Denote $\phi_k(y)= \psi_k(y)= (u,v)$ for each $k$, this implies that $(u,v)$ has diferent roles depending on each local chart. Figure \ref{fig50} is a geometric representation of the local coordinates of $(u,v)$ in each local chart. We observe that the points of $\mathbb{S}^1$ in each local chart have their coordinate $v=0$.
\begin{figure}
    \centering
    \includegraphics[scale=1]{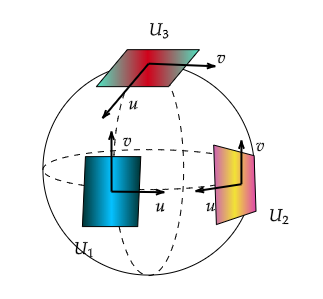}
    \caption{The local charts $(U_k,\phi_k)$ with $k=1,2,3$ for the Poincaré sphere.}
    \label{fig50}
\end{figure}
Consequently, if $\mathcal{X}(x)=(P(x_1,x_2),Q(x_1,x_2))$, then $\overline{X}(y)=Df^+(x)X(x)$ with $y=f^+(x)$ and 
$$
D\phi_1(y)\overline{X}(y)=D\phi_1(y)\circ Df^-(x)X(x)=D(\phi_1\circ f^+)(x)X(x).
$$
Then
$
(\phi\circ f^+)(x)=\left(\dfrac{x_2}{x_1},\dfrac{1}{x_1}\right)=(u,v),
$
and
$$
\overline{X}|_{U_1}:= D\phi_1(y)\overline{X}(y)=\left(\begin{array}{cc}
   \displaystyle {x_2}/{x_1} & \displaystyle {1}/{x_1} \vspace*{0.2cm} \\
  \displaystyle  -1/{x_1^2} & 0
\end{array}\right) \left(\begin{array}{c}
     P(x_1,x_2) \\
     Q(x_1,x_2)
\end{array}\right),
$$
$$
=\displaystyle{\frac{1}{x_1^2}(-x_2P(x_1,x_2)+Q(x_1,x_2),-P(x_1,x_2))},
$$
$$
=\displaystyle{v^2\left(-\frac{u}{v}P\left(\frac 1v,\frac uv\right)+\frac 1v Q\left(\frac 1v,\frac uv\right),-P\left(\frac 1v,\frac uv\right)\right)}.
$$
Moreover, if $m(z)=(1+u^2+v^2)^{\frac{1-d}{2}}$, as
$
\rho(y)=y_3^{d-1}=\dfrac{1}{\Delta(x)^{d-1}}=\dfrac{v^{d-1}}{\Delta(z)^{d-1}}=v^{d-1}m(z),
$
we have
$$
\rho(\overline{X}|_{U_1})(z)=v^{d+1}m(z)\left(-\frac{u}{v}P\left(\frac 1v,\frac uv\right)+\frac 1v Q\left(\frac 1v,\frac uv\right),-P\left(\frac 1v,\frac uv\right)\right).
$$
After a reparametrization of the time we eliminate the factor $m(z)$ from the compactified vector field and the analytic expression of the vector field $p(X)$ in the local chart $(U_1,\phi_1)$ is  
\begin{align} 
    \nonumber & \dot u = v^d\left( -uP\left(\frac{1}{v},\frac{u}{v}\right) +Q\left(\frac{1}{v},\frac{u}{v}\right)\right),\qquad
     \dot v = -v^{d+1}P\left(\frac{1}{v},\frac{u}{v}\right).
\end{align}

Analogously, in the local chart $(U_2,\phi_2)$ we have
\begin{align} 
    \nonumber & \dot u= v^d\left(P\left(\frac{u}{v},\frac{1}{v}\right) -uQ\left(\frac{u}{v},\frac{1}{v}\right)\right),\qquad \dot v= -v^{d+1}Q\left(\frac{u}{v},\frac{1}{v}\right),
\end{align}
and, in the chart $(U_3,\phi_3)$ 
\begin{align}\label{eq302}
   \nonumber & \dot u=P(u,v),\qquad \dot v= Q(u,v).
\end{align}

We remark that the expression of the vector field $p(X)$ in the remaning local charts $(V_k,\psi_k)$ is the same as $(U_k,\phi_k)$ except by multiplication by the factor $(-1)^{d-1}$, for $k=1,2,3$. 

The Poincaré disc $\mathbb D$ is the projection of the closed hemisphere $H^+$ under $(y_1,y_2,y_3) \to (y_1,y_2)$. In the Poincar\'e disc the \textit{finite singular points} (respectively, \textit{infinite}) of $X$ or $p(X)$ are the singular points of $p(X)$ that are in $\mathbb S^2\setminus \mathbb S^1$ (respectively, $\mathbb{S}^1$). It is important to remark that, if $y\in \mathbb{S}^1$ is an infinite singular point, then $-y$ is also an  infinite singular point and that the local phase portrait of $-y$ is the local phase portrait of $y$ multiplied by $(-1)^{d-1}$, it follows that the orientation of the orbits changes when the degree is even. Due to the fact that infinite singular points appear in pair of points diametrally opposite, it is enough to study the local phase portrait of only half of the infinite points, and using the degree of the vector field, it is possible to determine the other half. 

Finally we introduce the concept of topologically equivalent vector fields. Let ${\mathcal{X}}_1$ and ${\mathcal{X}}_2$ be two polynomial
vector fields on $\mathbb{R}^2$, and let $p({\mathcal{X}}_1)$ and $p({\mathcal{X}}_2)$ be their respective polynomial vector fields on the 
Poincar\'e disc $\mathbb{D}$. We say that they are \textit{topologically equivalent} if there exists a homeomorphism on the Poincar\'e disc $\mathbb{D}$ which preserves the infinity ${\mathbb S}^{1}$ and sends the orbits of $p({\mathcal{X}}_1)$  to orbits of $p({\mathcal{X}}_2)$, {\it preserving the orientation of all the orbits}.

\subsection{Li\'enard systems and limit cycles} 
One important kind of differential systems in the context of this paper are the so called Li\'enard systems, that is a differential equation of the form
\begin{equation*}
    \frac{d^2x}{dt^2}+f(x)\frac{dx}{dt}+g(x)=0,
\end{equation*}
where $f,g\in C^1(\mathbb{R})$. Applying the change of coordinate  $y=\dot{x} + F(x)$, where $F(x)= \int_0^x f(s) ds$, the Li\'enard equation is equivalent to the planar system
\begin{equation}\label{eq101a}
    \frac{dx}{dt}=-\phi(y)-F(x),\qquad \frac{dy}{dt}=g(x).
\end{equation}
We denote by $Lip$ the set of functions $f: I\subset \mathbb R \to \mathbb R$, such $f$ is a Lipschitz continuous function, that is, for each $f \in Lip$ there is a real constant $M$ such as $$ ||f(x) - f(y)|| \leq M ||x -y||,$$ for all $x,y \in I$, where $I$ is an interval of $\mathbb R$.

The next theorem was proved in \cite{1} and it will be one of the fundamental tools in the proof of existence and uniqueness of the limit cycles for the generalized Rayleigh systems.
 
 \begin{theorem}\label{teo101}
If the following conditions are satisfied for system \eqref{eq101a}
 \begin{enumerate}
     \item $g(x)\in$  Lip in any finite interval; $xg(x)>0, \text{ }x\neq 0;\text{ } G(-\infty)=G(+\infty)$, where $G(x)=\int_{0}^{x}g(s)ds$,
     \item $F'(x)=f(x)\in C^0(-\infty,+\infty);\text{ } F(0)=0;\text{ } \displaystyle\frac{f(x)}{g(x)}$ is non-increasing when $x$ grows in $(-\infty,0)$ and $(0,+\infty);\text{ }\displaystyle \frac{f(x)}{g(x)}$ is non-constant when  $0<|x|<<1,$
     \item $\phi(y)\in$ Lip in any finite interval, $y\phi(y)> 0,\text{ }y\neq 0, \text{ } \phi(y)$ is non-decreasing, $\phi(y)$ has left and right-derivatives, $\phi'_+(0)$ and $\phi'_-(0)$ in the point $y=0$, $\phi'_+(0).\phi'_-(0)\neq 0$ when $f(0)=0$.
    
 \end{enumerate}
then it has at most one limit cycle. If such limit cycle exists then it is stable.

\end{theorem}

\section{Proof of Theorem \ref{teo1}} 

In this section we investigate the local behavior of the Rayleigh system at infinity, i.e. near $\mathbb S^1$, the boundary of the Poincaré disc.

\begin{proposition}\label{prop70} The local phase portraits at the infinite singular points of system \eqref{eq2} in the Poincaré disc is topologically equivalent to the one described in Figure \ref{figurax}.
\end{proposition}

 \begin{figure}
    \centering
    \includegraphics[scale=0.4]{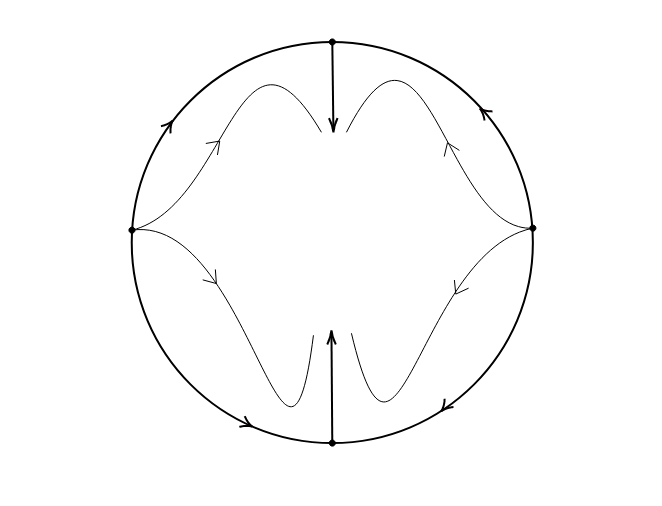}
    \caption{Case $a<0$.}
    \label{figurax}
\end{figure}

\begin{proof}
The Poincaré compactification of system \eqref{eq2} in the local chart $U_1$ is 
\begin{equation*}\label{eq306}
\begin{array}{l}
     \dot{u}=-a u+v^{2n}(au-1-u^2), \\
     \dot{v}=-a v+v^{2n+1}(a-u). 
\end{array}
\end{equation*}

Doing $v=0$ we conclude that the origin of the local chart $(U_1,\phi_1)$ is the unique singular point at this local chart and the Jacobian matrix in $(0,0)$ is
 $$
 \left[\begin{array}{cc}
 -a & 0\\
 0 & -a
 \end{array}\right].
 $$
 So the origin from the local chart  $U_1$ is an atractor node if $a>0$ and it is a repeller node if $a<0$.

In the local chart $(U_2,\phi_2)$ the compactified vector field \eqref{eq2} is given by
\begin{equation*}\label{eq307}
\begin{array}{l}
     \dot{u}=a u^{2n+1}+v^{2n}(1-au+u^2), \\
     \dot{v}= u v^{2n+1}
\end{array}
\end{equation*}
Then the origin of the local chart $U_2$ is a degenerate singular point and, to investigate its local phase portrait we apply the vertical directional \textit{blow up}, see \cite{art7} for more details about blow ups. Using the change of variables $u=x$ and $v=zx$, and the reparametrization of time, that eliminate the common factor $x^{2n-1}$ of the two components of $\dot x$ and $\dot z$, we get the system 
$$
  \begin{array}{l}
       \dot{x}= ax^2+xz^{2n}(1-ax+x^2), \\
       \dot{z}=-z(ax+z^{2n} -az^{2n}x),
  \end{array}
 $$
 Where, again, the origin is the unique singular point of the system and it is degenerate so we apply the \textit{blow up} once again, this time we take the change of coordinates $x=w^{2n}r$ and $z=w$, that is equivalent to sucessive blow ups. After a reparametrization of the time that eliminate the common factor $w^{2n}$ from both equations, the system is written as
 \begin{equation} \label{eq2000}
     \begin{array}{l}
           \dot r = r((2n+1)(1+ar) -w^{2n}(a(2n+1))r+r^2w^{2n}),\\
           \dot w= w(-1+ a r(-1+w^{2n}))
           \end{array}
 \end{equation}
 
 Notice that the system (\ref{eq2000}) has two critical points, $(0,0)$ and $(-1/a,0)$. The Jacobian matrix at $(0,0)$ is
 $$
 \left[\begin{array}{cc}
 1+2n & 0\\
 0 & -1
 \end{array} \right],
 $$
 and so $(0,0)$ is a saddle point. The Jacobian matrix at $(-1/a,0)$ is
  $$
 \left[\begin{array}{cc}
-(1+2n) & 0\\
 0 & 0
 \end{array} \right],
 $$
 then the critical point $(-1/a,0)$ is a semi-hyperbolic point. After the reparametrization $r\rightarrow s-1/a,t\rightarrow -t$, the conditions of Proposition \ref{semihyperbolic} are satisfied, so, checking that $m$ is odd and $a_m<0$, we conclude that $(-1/a,0)$ is a saddle point. Applying the \textit{blowing down}, as represented in  Figure \ref{figura30},  we conclude the local behaviour of the origin of the local chart $U_2$.
 \begin{figure}
     \centering
     \includegraphics[scale=0.5]{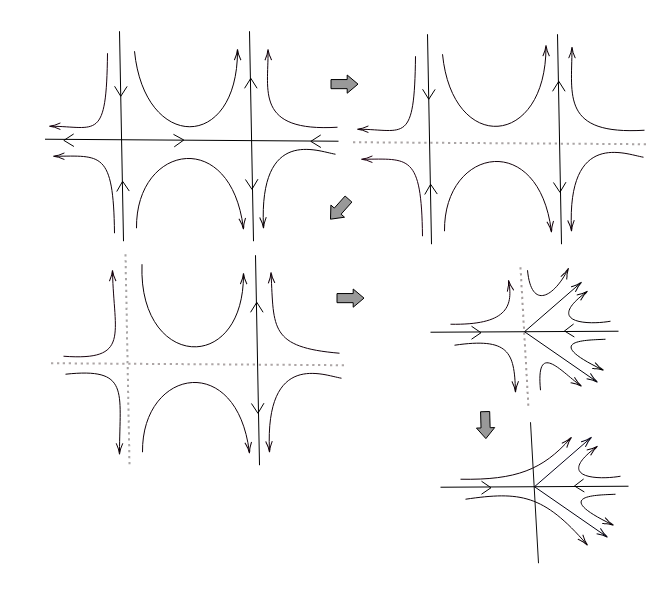}
     \caption{\textit{Blow up} process to the system (\ref{eq2000}) when $a<0$.}
     \label{figura30}
 \end{figure}

 Bringing together the local information about the infinite singular points in both local charts and using the continuity of the solutions we conclude that the orbits of system \eqref{eq2} near of $\mathbb S^1$ are topologically equivalent to the one described in Figure \ref{figurax}, except by the reversing of all orbits when $a\neq 0$. See that for $a=0$ system \eqref{eq1} is a linear center. \end{proof}
 
 \begin{remark}
 We point out that in \cite{art2} the statement equivalent to this last proposition has a misprint. The authors say that if $a>0$ the infinity singular point in the local chart $U_1$ is repeller when in fact it is an atractor. The same for $a<0$ to be an atractor when it is a repeller. 
 \end{remark}
 
 The next result describes the local phase portrait of system \eqref{eq2} at the finite singular points.
 
 \begin{lemma} \label{lemma:finite}
  System \eqref{eq2} has a unique finite singular point, the origin which is a stable node if $a\geq 2$, an unstable node if $a\leq -2$, a unstable focus if $-2<a<0$, a stable focus if $0 < a < 2$ and a center if $a=0$.
\end{lemma}

\begin{proof} We can easily verify that the unique critical point of system \eqref{eq2} is the origin and the Jacobian matrix of system \eqref{eq2} at the origin is given by 
$$
\left[\begin{array}{cc}
-a &	1\\
-1 &	0
\end{array}\right].
$$
So the lemma follows from the local study described in Section \ref{preliminar}.
\end{proof}

Next theorem guarantee the existence of at least one limit cycle for system \eqref{eq2}. 

\begin{theorem}\label{teo:ciclo}
The generalized Rayleigh system \eqref{eq2} has exactly one limit cycle if $a\neq 0$. 
\end{theorem}

\begin{proof}
Assume that $a<0$ then from Lemma \ref{lemma:finite} there exists a unique singular point of system \eqref{eq2}, the origin and it is a repellor. Moreover, from Proposition \ref{prop70}, we conclude that each solution of system \eqref{eq2} is moving away from $\mathbb S^1$ in the Poincar\'e disc.

As the Poincar\'e disc is a compact set with a unique singular point, it follows from the Poincar\'e--Bendixson Theorem (Theorem \ref{semihyperbolic}) that there exists at least one limit cycle.

The study of the case $a>0$ is analogous.
\end{proof}

Now we have each ingredient to prove Theorem \ref{teo1}.

\begin{proof}[Proof of the Theorem \ref{teo1}.] Assuming $a>0$ and applying in \eqref{eq2} the change of coordinates $(x,y,t) \rightarrow (x,y,-t)$ (that only changes the orientation of the orbits), we obtain $$ \dot{x}= -y +a(1-x^{2n})x, \qquad \dot{y}=x.$$

Taking $g(x)=x$, it is immediate that $g \in Lip$ in any finite interval of the real line and $xg(x)>0$, for $x\neq 0$. Moreover, 
$$
\int_0^{-\infty}sds=\int_0^{+\infty}sds=+\infty.
$$
If $f(x)=F'(x)=-a(1-x^{2n})+2n\,a\,x^{2n}=-a+a(2n+1)x^{2n}$, then $f(x)$ is a real valued continuous function defined in the real line such that $F(0)=0$. Therefore 
$$\dfrac{f(x)}{g(x)}=\dfrac{-a+a(2n+1)x^{2n}}{x}$$ is a non-decreasing real function in the intervals  $(-\infty,0)$ and $(0, +\infty)$ because the derivative of $f(x)/g(x)$ for $x\ne 0$ is positive, and also it is a non-constant function in any neighborhood of the origin. 

When $\phi(y)=y$, \, $y\phi(y)=y^2>0$ for all values of $y\neq 0$,  $\phi(y)\in Lip$ in any finite interval of the real line and $\phi(y)$ is non-decreasing function. 

In short, each one of the hypothesis of Theorem \ref{teo101} is satisfied and the existence of a unique limit cycle to system \eqref{eq2} when $a>0$ is guaranted. 

The same result is valid when $a<0$ because system \eqref{eq2} with $a>0$ is topologically equivalent to system \eqref{eq2} with $a<0$, through the change of coordinates $(x,y,t) \rightarrow (-x,y,-t)$. So Theorem \ref{teo1} is proved.
\end{proof}

\begin{remark}  The existence of at least one limit cycle for system \eqref{eq2} when the parameter $a\ne 0$ is sufficiently small was proved in \cite{art3}, but for proving this result the authors use the averaging theory of first order that does not guarantee the uniqueness of the limit cycle, it only guarantee the existence of at least one of such periodic orbit. Using Theorem \ref{teo101} from \cite{1} we have the uniqueness of the limit cycles and the desired result for any non zero $a \in \mathbb R\setminus \{0\}$ and for any positive integer $n$. 
\end{remark}

The proof of Theorem \ref{teo2} is a consequence from the following results: Proposition \ref{prop70},  Lemma \ref{lemma:finite}, Theorem \ref{teo:ciclo} and \ref{teo1}.

\section*{Acknowledgments} 
The first author is partially supported by Coordena\c{c}\~ao de Aperfei\c{c}oamento de Pessoal de N\'ivel Superior - Brasil (CAPES) - Finance Code 001. The second author is partially supported by the Ministerio de Ciencia, Innovaci\'on y Universidades, Agencia Estatal de Investigaci\'on grant PID2019-104658GB-I00, the Ag\`encia de Gesti\'o d'Ajuts Universitaris i de Recerca grant 2017SGR1617, and the H2020 European Research Council grant MSCA-RISE-2017-777911. The third author is partially supported by Projeto Tem\'atico FAPESP number 2019/21181--0 and by Bolsa de Produtividade-CNPq number 304766/2019--4.

\section*{Data availability}
Data sharing is not applicable to this article as no new data were created or analyzed in this study.
	
\bibliographystyle{acm}	
\bibliography{references}


\end{document}